\documentclass[11pt]{article}
\usepackage{amsfonts}
\usepackage{amssymb}

\usepackage{amsmath,amsthm,enumerate}
\usepackage{color}
\usepackage[pdftex,colorlinks=true,bookmarks=false,citecolor=blue]{hyperref}

=0pt
\def\sqr#1#2{{\vcenter{\vbox{\hrule height.#2pt
              \hbox{\vrule width.#2pt height#1pt \kern#1pt \vrule
width.#2pt}
              \hrule height.#2pt}}}}
\def\3n{\negthinspace \negthinspace \negthinspace }
\def\2n{\negthinspace \negthinspace }
\def\1n{\negthinspace }

\def\no{\noindent}

\def\bs{\bigskip}

\def\({\Big (}
\def\){\Big )}
\def\[{\Big[}
\def\]{\Big]}

\def\={\buildrel \triangle \over =}

at 9pt 

\def\be{\begin{equation}}
\def\bel{\begin{equation}\label}
\def\ee{\end{equation}}
\def\bea{\begin{eqnarray}}
\def\eea{\end{eqnarray}}
\def\bt{\begin{theorem}}
\def\et{\end{theorem}}
\def\bc{\begin{corollary}}
\def\ec{\end{corollary}}
\def\bl{\begin{lemma}}
\def\el{\end{lemma}}
\def\bp{\begin{proposition}}
\def\ep{\end{proposition}}
\def\br{\begin{remark}}
\def\er{\end{remark}}
\def\ba{\begin{array}}
\def\ea{\end{array}}
\def\bd{\begin{definition}}
\def\ed{\end{definition}}

\newtheorem{lemma}{Lemma}[section]
\newtheorem{remark}{Remark}[section]
\newtheorem{example}{Example}[section]
\newtheorem{theorem}{Theorem}[section]
\newtheorem{corollary}{Corollary}[section]

\newtheorem{definition}{Definition}[section]
\newtheorem{proposition}{Proposition}[section]

\makeatletter
   
   \@addtoreset{equation}{section}
\makeatother

\allowdisplaybreaks

\begin{document}

\title{\bf Numerical Integration at the Intersection of Probability Theory, Number Theory and Dynamical Systems\thanks{The research of the first and third authors is partially supported by NSF of China under grant 11931011, New Cornerstone Investigator Program, and the Science Development Project of Sichuan University under grant 2020SCUNL201. The research of the second author is partially supported
by the Key Project of the Sichuan Science and Technology Education Joint Fund under grant 2024NSFSC1963  and the National Natural Science Foundation of China  under grant 12071324.}}

\author{Zicheng Wang\thanks{School of Mathematics, Sichuan University, Chengdu 610064,  China.
{\small\it E-mail:} {\small\tt
zicheng$\_$wang@yeah.net}.},~~~ Haisen Zhang\thanks{School of Mathematical Sciences, Sichuan Normal University, Chengdu  610068, China. {\small\it E-mail:}
{\small\tt haisenzhang@yeah.net}.}~~~ and~~~ Xu
Zhang\thanks{School of Mathematics, Sichuan
University, Chengdu 610064,
China.  {\small\it
E-mail:} {\small\tt zhang$\_$xu@scu.edu.cn}.}}

\date{}

\maketitle

\begin{abstract}
This paper is devoted to the study of uniformly distributed (u.d.) sequences on general non-compact topological spaces. Convergence-determining classes and the equivalent characterizations of u.d.\ sequences on $T_4$ spaces are discussed. For Polish spaces  endowed with regular probability measures, the existence of a countable convergence-determining class for u.d.\ sequences is proved by virtue of the tightness of Borel probability measures. It is shown that, under the condition that a countable convergence-determining class exists, both independent identically distributed (i.i.d.) randomly sampled sequences and orbit sequences of ergodic transformations are  u.d.\ with full measure.
\end{abstract}

\bs

\no{\bf 2010 Mathematics Subject
Classification:} Primary 65D30, 11K36, secondary 11K45,28D05.

\bs

\no{\bf Key Words:} Uniform distribution of sequences, Polish  spaces, numerical integration, countable convergence-determining class, ergodic transformation.

%\newpage

\section{Introduction}\label{s1}

Let $(X,\tau)$ be a topological space and $\mathcal{B}(X)$ be the corresponding Borel $\sigma$-algebra generated by this topology. Let $\mu$ be a given measure on $(X, \mathcal{B}(X))$  and $f: X \to \mathbb{R}$ a given function. In numerical integration, one often aims to construct a sequence $\{x_n\}_{n=1}^{\infty}$ in $X$ such that:
\begin{equation}\label{eq 1.1}
\lim_{N \to \infty} \left| \frac{1}{N} \sum_{k=1}^{N} f(x_k) - \int_X f(x) d\mu \right| = 0.
\end{equation}
If the sequence $\{x_n\}_{n=1}^{\infty}$ satisfies \eqref{eq 1.1} for all bounded continuous functions $f$ on $X$, then $\{x_n\}_{n=1}^{\infty}$ is called a u.d. sequence on $X$ with respect to the measure $\mu$.

Uniformly distributed sequences form the essential nexus linking three fundamental methods in numerical integration: the Monte Carlo method, the quasi-Monte Carlo method, and the ergodic method. If $\mu$ is a probability measure, Monte Carlo methods obtain u.d.\ sequences by randomly generating i.i.d.\ random samples \cite{LimicNedzad2018}. More precisely, using the Strong Law of Large Numbers (LLN) and Weyl criterion, the sample paths $\{x_n\}_{n=1}^{\infty}$ of an i.i.d.\ random variable sequence $\{X_n\}_{n=1}^{\infty}$ are u.d.\ sequences with probability 1. The Central Limit Theorem can be used to prove that the Monte Carlo numerical integration method has a convergence rate of $O(1/\sqrt{N})$ in the sense of probability \cite{caflisch1998monte}, i.e.,
$$
\lim_{N \to \infty} \mu \left( \left| \frac{1}{N} \sum_{k=1}^{N} f(x_k) - \int_X f(x) d\mu \right| \leq \frac{x_{\alpha/2} \sigma(f)}{\sqrt{N}} \right) = 1 - \alpha.$$
Here, $N$ is the sample size,  $$\sigma(f) = \left[ \int_X f^2(x) d\mu - \left( \int_X f(x) d\mu \right)^2 \right]^{1/2}$$ is the standard deviation of $f$, $1 - \alpha$ ($\alpha \in (0, 1)$) is the confidence level, and $x_{\alpha/2}$ is the $\alpha/2$-quantile of the standard normal distribution. The advantage of the Monte Carlo method is that, theoretically, its convergence rate is independent of the dimension of the underlying space of the arguments of the integrand.
However, in practice, truly i.i.d.\ samples are computationally unattainable. Popular random number generators, including Linear Congruential Generators (LCG) \cite{HullDobell1962} and the Mersenne Twister \cite{Mersenne_Twister1998}, yield long-period pseudorandom sequences, which are only suitable for finite-dimensional numerical integration.

Quasi-Monte Carlo methods construct u.d.\ sequences with low-discrepancy through number-theoretic approaches. On the finite-dimensional unit cube $[0, 1]^d$ ($d \in \mathbb{N}$), the discrepancy of a finite sequence ${x_1, x_2, \dots, x_N}$ is evaluated by the star discrepancy $D_N^*$, where
$$D_N^* (x_1, x_2, \dots, x_N) = \sup_{a \in [0, 1)^d} \left| \frac{A([0, a); N)}{N} - \text{Vol}([0, a)) \right|.
$$
Here, $[0, a) = \prod_{k=1}^d [0, a_k)$ ($a = (a_1, a_2, \dots, a_d)^{\top} \in [0, 1)^d$), $A([0, a); N)=\sum_{n=1}^N \chi_{[0, a)}(x_n)$ with $ \chi_{[0, a)}$ being the characteristic function of $\chi_{[0, a)}$, and $\text{Vol}([0, a)) = \prod_{k=1}^d a_k$ is the volume of $[0, a)$. A sequence $\{x_n\}_{n=1}^{\infty}$ is called a low-discrepancy sequence if its star discrepancy is $O((\log N)^d / N)$ as $N \to \infty$ \cite{niederreiter1992random}.

For common low-discrepancy sequences, such as Sobol sequences \cite{sobol1967distribution}, Faure sequences \cite{faure1981discrepance}, Niederreiter sequences \cite{niederreiter1987point}, and Halton sequences \cite{Halton1960}, the decay order of their star discrepancy $D_N^*$ is $O((\log N)^d / N)$, which exhibits better uniformity than pseudorandom numbers in Monte Carlo methods \cite{caflisch1998monte}. By Koksma-Hlawka inequality \cite{KoksmaHlawka1962}, the numerical integration error of the Quasi-Monte Carlo method satisfies
$$\left| \frac{1}{N} \sum_{k=1}^N f(x_k) - \int_{[0, 1]^d} f(x) dx \right| \leq \text{Var}(f) \cdot D_N^* (x_1, \dots, x_N) = O\left( \frac{(\log N)^d}{N} \right).
$$
Here, $\text{Var}(f)$ is the total variation of $f$. Provided that the integrand $f$ is smooth enough, the convergence order of numerical integration can be further improved by optimally designing the construction strategy for u.d.\ sequences \cite{dick2008arbitrary}. A comprehensive review of the basic theory and construction methodologies for low-discrepancy sequences can be found in \cite{Kritzer2014book,niederreiter2017resent}. Notably, existing low-discrepancy sequences constructed via number-theoretic approaches  exhibit dimension-dependent convergence rates, making them  suitable only for moderate-dimensional cases; furthermore, their construction methodologies cannot be straightforwardly extended to infinite-dimensional spaces.

The ergodic method generates u.d.\ sequences via orbits of ergodic transformations, with its theoretical foundation being Birkhoff's Ergodic Theorem \cite{BirkhoffKoopman1932}. Specifically, for a probability measure space $(X, \mathcal{B}(X), \mu)$ ($\mu(X) = 1$), if there exists an ergodic transformation $T: X \to X$, then by appropriately choosing an initial point $x \in X_0$, the orbit sequence $\{T^n(x)\}_{n=0}^{\infty}$ of $T$ is a $\mu$-u.d.\ sequence. Under appropriate conditions, the Birkhoff Ergodic Theorem and Weyl criterion can be used to prove that the set of all such initial points has full measure. Deterministic trajectory sequences generated by ergodic transformations avoid the correlation artifacts inherent to pseudorandom number generators in Monte Carlo integration, thus exhibiting significantly improved numerical stability compared to conventional Monte Carlo schemes. In contrast to quasi-Monte Carlo methods, ergodic approaches impose much weaker structural constraints on the underlying space. From a theoretical standpoint, the only fundamental requirement for the ergodic method is the existence of an ergodic transformation on the underlying measure space.

A limitation of the ergodic method lies in the uncertainty of its convergence rate. In fact, although Birkhoff's Ergodic Theorem guarantees the convergence of numerical integration using orbit sequences of ergodic transformations, the convergence rate depends on the structural characteristics of the ergodic transformation and the properties of the integrand \cite{kachurovskii2016estimates}. Under suitable structural assumptions on the underlying spaces, certain low-discrepancy sequences employed in quasi-Monte Carlo methods can be interpreted as orbit sequences generated by ergodic transformations \cite{GrabnerHellekalek2012}, enabling the discrepancy of these orbit sequences to achieve $O((\log N)^d / N)$. However, for other ergodic transformations, the convergence rate slows down significantly \cite{bayart2020fast,kachurovskii1996rate,kachurovskii2016estimates}. In  \cite{bayart2020fast}, the authors proved that there exist an ergodic transformation $T$ and a continuous function $f$ such that the convergence rate of the time average $\frac{1}{N} \sum_{n=0}^{N-1} f(T^n(x))$ can be arbitrarily slow. This conclusion implies that the ergodic method lacks a unified guarantee for its convergence rate, requiring separate analysis for specific transformations and functions, which significantly restricts its applicability in practical computational scenarios. In \cite{DasSaikiel2017,DasYorke2018,tong2024exponential}, it is shown that under appropriate regularity assumptions on the integrand, the weighted Birkhoff averages can theoretically achieve arbitrary polynomial convergence rates or even exponential convergence rates.

The theory of uniform distribution originated from the study of u.d.\ sequences modulo 1 \cite{weyl1916gleichverteilung}. The famous Weyl criterion provides a fundamental tool for determining whether a sequence is u.d.\ It reduces the convergence test over all bounded continuous functions on the interval $[0, 1]$ to a convergence test over a countable set of special bounded continuous functions(the normalized complex exponential basis of Fourier series). Following Weyl's work, the theory of uniform distribution has been generalized  in multiple directions, such as the theory of uniform distribution on compact Hausdorff spaces \cite{hlawka1956folgen, hlawka1958folgen}, compact groups \cite{Eckmann1943,helmberg1958theorem,veech1971some,Losert1978existence}, locally compact groups \cite{Benzinger1974I,Benzinger1974II,LosertRindler1978,Rubel1965LocallyComp}, and the theory of uniform distribution with respect to signed measures on compact spaces \cite{mercourakis2023uniform}. For a comprehensive overview of the classical theory of u.d.\ sequences, we refer the readers to the monograph by Kuipers and Niederreiter \cite{kuipersNiederreiter1974}.

In physics, engineering, statistics, finance, and other fields, there are numerous examples requiring numerical integration on infinite-dimensional topological linear spaces. As elaborated above, the central objective of numerical integration is the construction of u.d.\ sequences characterized by low-discrepancy. Therefore, investigating the existence, characterization, discrepancy, and construction methods of u.d.\ sequences on infinite-dimensional topological linear spaces is of significant theoretical and practical importance. However, to the best of our knowledge, a theoretical framework for uniform distribution on infinite-dimensional spaces has not yet been established, leading to a lack of unified theoretical analysis and efficient numerical implementation methods for numerical integration on infinite-dimensional spaces. Although some scholars have attempted to explore numerical integration methods in infinite-dimensional spaces, see for instance, \cite{dick2013high,pantsulaia2016book}, these efforts are limited to special cases such as the Hilbert cube $[0, 1]^{\infty}$ and adopt finite-dimensional approximation strategies, failing to achieve a fundamental breakthrough beyond the existing theoretical framework for finite-dimensional spaces.

One main obstacle to studying u.d.\ sequences on infinite-dimensional spaces is the lack of compactness in such spaces. The key role of compactness is to ensure the existence of a countable convergence-determining class, enabling uniform distribution verification via convergence tests on the class. This embodies the central principle of Weyl  criterion and serves as an indispensable tool for investigating the u.d.\ sequences. On compact Hausdorff spaces, the existence of a countable convergence-determining class can be easily proven using the Stone-Weierstrass theorem. However, on non-compact spaces, the Stone-Weierstrass theorem fails, making the study of the existence of countable convergence-determining classes and related theory of u.d.\ sequences inherently challenging.

To address the difficulties in the theory of uniform distribution on infinite-dimensional spaces, this paper systematically establishes a theoretical framework for u.d.\ sequences on general non-compact topological spaces. First, we discuss convergence-determining classes and the equivalent characterizations of $\mu$-u.d.\ sequences on $T_4$ spaces. Second, we prove the existence of a countable convergence-determining class for u.d.\ sequences on Polish spaces. Based on this conclusion, we demonstrate that, in the sense of full measure, random sampling sequences and orbits of ergodic transformations are u.d.\ sequences. These results not only confirm the existence and construction of u.d.\ sequences on general non-compact topological spaces but also further support theoretical research into the discrepancy of u.d.\ sequences and into numerical integration methods in infinite-dimensional topological linear spaces.

The rest of this paper is organized as follows. In Section 2, we present some basic concepts and results required for the subsequent sections. In Section 3, we discuss the convergence-determining class and the equivalent characterizations of $\mu$-u.d.\ sequences on $T_4$ spaces. In Chapter 4, we study the existence of a countable convergence-determining class for u.d.\ sequences on Polish spaces, and based on this, discuss the practicability of constructing  u.d.\ sequences by random sampling method and ergodic transformation method on Polish spaces. Finally, in Chapter 5, we provide concluding remarks and propose several topics for future research based on the results of this paper.

\section{Preliminaries}
	In this section, we collect some basic notions and results from topology, measure theory, and theory of uniform distribution. For more details we refer the readers to \cite{Durrett2019Prob,Kechris1995,kuipersNiederreiter1974,Munkres2000Topology}.

Let $(X,\tau)$ be a topological space with topology $\tau$. If $\tau$ is metrizable, we let  $d$  denote the corresponding metric of $\tau$. $(X, \tau)$ is {\em $T_1$} if, for any distinct points $x, y \in X$, there exist open sets $U, V \in \tau$ such that $x \in U\setminus V $,  $y \in V\setminus U$; it is  {\em Hausdorff (or $T_2$)} if for any distinct points $x, y \in X$ there exist disjoint open sets $U, V \in \tau$ such that $x \in U$ and $y \in V$; it is \emph{normal} if any disjoint closed sets $A, B \subseteq X$, there exist disjoint open sets $U, V \in \tau$ such that $A \subseteq U$, $B \subseteq V$. A normal space that is also $T_1$ is called a \emph{$T_4$ space}. $(X, \tau)$  is a Polish space if its topology $\tau$ can be induced by a complete metric on $X$ and  contains a countable dense subset.

We denote by  $b(X)$ the set of bounded real-valued Borel measurable functions on $X$, equipped with the supremum norm $\|f\|_\infty = \sup_{x \in X} |f(x)|$, making $b(X)$ a Banach space, and even a Banach algebra with usual algebraic operations for functions. The subset $\mathcal{C}_b(X) \subseteq b(X)$ consists of all bounded continuous real-valued functions, forming a Banach subalgebra.

\begin{lemma}(Urysohn's Lemma, \cite[Theorem 33.1]{Munkres2000Topology})\label{lem:urysohn}
If $(X, \tau)$ is a normal space, then, for any disjoint closed sets $E, F \subseteq X$, there exists a continuous function $h: X \to [0,1]$ such that $h(x) = 0$ on $E$ and $h(x) = 1$ on $F$.
\end{lemma}
	
\begin{lemma}(Tietze Extension Theorem,
\cite[Theorem 35.1]{Munkres2000Topology})\label{lem:tietze} Let $(X, \tau)$ be a normal space and $A \subseteq X$ a closed subset. Any continuous function of $A$ into the closed interval $[a,b]\subset R$ may be extended to a continuous function of all $X$ into $[a,b]$.
\end{lemma}

The following result follows from the Stone-Weierstrass Theorem (See, for instance \cite[Theorem 44.5]{Willard1970topology}).

\begin{lemma}(\cite[P. 245]{Kelley1955})\label{lem:separable_Cb}
Let $K$ be a compact separable metric space. Then, the Banach space $\mathcal{C}_b(K)$ is separable.
\end{lemma}

\begin{definition}(\cite[P. 211]{Munkres2000Topology})\label{def:separation_closed_sets}
Let $\mathcal{V}$ be a subalgebra of $\mathcal{C}_b(X)$. We say that $\mathcal{V}$ separates closed sets if for any two disjoint closed sets $A, B \subset X$, there exists a function $h \in \mathcal{V}$ such that $0 \leq h(x) \leq 1$ for any $x\in X$, and $h(x)= 0$ on $A$, and $h(x)= 1$ on $B$.
\end{definition}

By the Urysohn Lemma, when $(X,\tau)$ is a normal space, for any  pair of disjoint closed sets $A, B \subset X$, there exists a continuous mapping $h:X\to [0,1]$ that separates $A$ and $B$. Consequently, when $\mathcal{V}$ contains all such Urysohn functions, $\mathcal{V}$ separates closed sets in $(X,\tau)$.

Denote by $\mathcal{B}(X)$ the Borel $\sigma$-algebra on $X$ and by $\mu$ a regular, non-atomic probability measure on $(X, \mathcal{B}(X))$. For sufficient topological conditions on spaces that guarantee the existence of non-trivial regular Borel measures, see \cite{Gardner1975}.

A subset $E \subseteq X$ is a  $G_\delta$ set  if it is the intersection of countably many open sets, i.e., $E = \bigcap_{n=1}^\infty U_n$ for some $\{U_n\}_{n=1}^\infty \subseteq \tau$.   By the  regularity of $\mu$, we have, for every $E \in \mathcal{B}(X)$,
$$\mu(E) = \sup \{ \mu(C)  \,|\,  C \subseteq E, C \text{ closed} \} = \inf \{ \mu(D) \,|\, E \subseteq D, D \text{ open} \}.$$
A Borel set $M \in \mathcal{B}(X)$ is called a \emph{$\mu$-continuity set} if $\mu(\partial M) = 0$, where $\partial M = \overline{M} \setminus \text{int} M$ is the boundary of $M$, with  $\text{int} M$ and $\overline{M}$ representing the interior and closure of $M$, respectively.

\begin{example}\label{ex:mu-continuous set}
Suppose that $g: X \to [0,1]$ is a continuous function. For each $\alpha \in [0,1]$, let $G_\alpha = \{ x \in X : g(x) = \alpha \}$. Since
$X = \bigcup_{\alpha \in [0,1]} G_\alpha$, and $\mu$ is a finite measure, by \cite[Exercise 1.21, Ch. 3]{kuipersNiederreiter1974}, there are at most countably many $\alpha \in [0,1]$ such that $\mu(G_\alpha) > 0$. Consequently, there exists $\alpha_0 \in (0,1)$ with $\mu(G_{\alpha_{0}}) = 0$. Define $C = \{ x \in X : g(x) \geq \alpha_{0} \}$. The set $C$ is closed, and its boundary satisfies $\partial C \subseteq G_{\alpha_{0}}$, implying that $\mu(\partial C) = 0$ and $C$ is a $\mu$-continuity set.
\end{example}

\begin{lemma}\cite[Theorem 17.11]{Kechris1995}\label{lem:tight-prob}
If $(X,\tau)$ is a Polish space, then any  probability measure $\mu$ on $(X,\mathcal{B}(X))$ is tight, that is,  for any   $E \in \mathcal{B}(X)$,
$$\mu(E)=\sup \{\mu(K)\,|\, K \subseteq E, K \text{ compact}\}.$$
Especially, for  every $\epsilon > 0$, there exists a compact set $K \subset X$ such that $\mu(K) > 1 - \epsilon$.
\end{lemma}

Two measurable spaces $(Y, \Sigma)$ and $(Z,\Xi)$ are said to be
{\em isomorphic} if there is a bijection $\varphi: Y \to Z$ such that both $\varphi$ and $\varphi^{-1}$ are measurable, moreover, if both $\Sigma$ and $\Xi$ are Borel $\sigma$-algebras, we call $(Y, \Sigma)$ and $(Z,\Xi)$ Borel isomorphic.
A measurable space $(Y, \Sigma)$ is said to be {\em standard} if there
exists a complete, separable metric space $S$ such that $(Y, \Sigma)$ is isomorphic (as a measurable space) to $(S,\mathcal{B}(S))$ with $\mathcal{B}(S)$ being the Borel $\sigma$-algebra of $S$. A measure space $(Y, \Sigma, \nu)$ is said to be
standard if $(Y, \Sigma)$ is a standard measurable space.
	
\begin{lemma}(\cite[Theorem 17.41]{Kechris1995}) \label{lem:atomless_isomorphism}
Let $(Y, \Sigma, \nu)$ be a standard measure space with probability measure $\mu$ such that $\nu(\{x\}) = 0$ for all $x \in X$. Then there is a Borel isomorphism $\varphi:(Y, \Sigma)\to ([0, 1], \mathcal{B}([0,1])$ with $\mu\circ \varphi^{-1}=\lambda|_{[0, 1]}$,  where $\mu\circ \varphi^{-1}(B)=\mu(\varphi^{-1}(B))$  for any $B\in \mathcal{B}([0,1])$ and $\lambda$ is the one-dimensional Lebesgue measure.
\end{lemma}

\begin{lemma}(Cavalieri's Principle,
\cite[Corollary 2.2.34]{willem2023functional})\label{lem:layer_cake}
Let $(Y, \Sigma, \nu)$ be a measure space. For any integrable function $\phi: Y \to [0, \infty]$,
$$
\int_Y \phi \, d\nu = \int_0^\infty \nu(\{ y \in Y \mid \phi(y) > t \}) \, dt.
$$
\end{lemma}

Let $(Y, \Sigma, \nu)$ be a measure space and $T: Y \to Y$ a measurable transformation. We say that $T$ is a {\em measure-preserving transformation} if $\nu(T^{-1}(A)) = \nu(A)$ for every $A \in \Sigma$.
Measure-preserving transformation $T: Y \to Y$ is said to be {\em ergodic} if, for every $A \in \Sigma$, $T^{-1}(A) = A$ implies  $\nu(A) = 0$ or $\nu(A) = 1$.

\begin{lemma}(Birkhoff Ergodic Theorem, \cite[Theorem 1.1.4]{walters1982})\label{lem:birkhoff}
Let $(Y, \Sigma, \nu)$ be a measure space with $\nu(X) = 1$, and $T: Y \to Y$ a measure-preserving transformation . If $T$ is ergodic, then for every $f \in L^1(Y, \Sigma, \nu)$,
$$\lim_{N \to \infty} \frac{1}{N} \sum_{n=0}^{N-1} f(T^n(y)) = \int_Y f \, d\nu$$
for $\nu$-a.e. $y \in Y$.
\end{lemma}	
	
\begin{lemma}(Strong LLN, \cite[Theorem 2.4.1]{Durrett2019Prob}) \label{lem:strong_law}
Let $(\Omega, \mathcal{F}, P)$ be a probability space and $\{X_n\}_{n=1}^\infty$ a sequence of i.i.d.\ random variables with $\mathbb{E}|X_1|=\int_{\Omega}|X_1|dP < \infty$. Then,
$$  \lim_{n \to \infty} \frac{1}{n} \sum_{k=1}^n X_k = \mathbb{E} X_1, a.s.$$
\end{lemma}

\section{Characterizations of u.d.\ sequences on  $T_4$  spaces}

In this section, we always assume that $(X,\tau)$ is a $T_4$ space  and $\mu$ is a regular, non-atomic probability measure on $(X, \mathcal{B}(X))$.
Analogous to the case of compact spaces (see, for instance, \cite{kuipersNiederreiter1974}), in this section we present the definition and equivalent characterizations of u.d.\ sequences on $(X,\mathcal{B}(X),\mu)$.

\begin{definition}
A sequence $\{x_n\}_{n=1}^\infty$ in $X$ is  $\mu$-u.d.\  ($\mu$-u.d.\) if, for every $f \in \mathcal{C}_b(X)$,
\begin{equation*}
\lim_{N \to \infty} \frac{1}{N} \sum_{n=1}^N f(x_n) = \int_X f \, d\mu.
\end{equation*}

  A class $\mathcal{V} \subseteq b(X)$ is  convergence-determining  with respect to $\mu$ if the condition
$$\lim_{N \to \infty} \frac{1}{N} \sum_{n=1}^N f(x_n) = \int_X f \, d\mu \quad \text{for all } f \in \mathcal{V}$$
implies $\{x_n\}_{n=1}^\infty$ is $\mu$-u.d.\
\end{definition}

Let $ \mathcal{V} \subseteq b(X) $. We denote by $\text{sp}(\mathcal{V})$ the linear subspace of $b(X)$ generated by $\mathcal{V}$, and $\overline{\text{sp}(\mathcal{V})}$ its closure under $\|\cdot\|_\infty$.	  By
\cite[Theorem 1.1, Ch.3]{kuipersNiederreiter1974}, if $ \mathcal{C}_b(X) \subseteq \overline{\text{sp}(\mathcal{V})} $, then $ \mathcal{V} $ is a convergence-determining class with respect to $ \mu $. When the space $X$ is compact, by the classical Stone-Weierstrass theorem(\cite[Theorem 44.5]{Willard1970topology}), a subalgebra $\mathcal{V}$ of $\mathcal{C}_b(X)$ is a convergence-determining class with respect to $ \mu $ if $\mathcal{V}$ contains the constant functions and separates points in $X$, a detailed proof can be found in
\cite[Corollary 1.1, Ch. 3]{kuipersNiederreiter1974}.
However, in non-compact spaces, separate  points in $X$ is not enough to ensure that $\mathcal{V}$   is a convergence-determining class. Instead, the following result says that the condition of separating closed sets (see Definition \ref{def:separation_closed_sets}) can be used to ensure that $\mathcal{V}$  is a convergence-determining class. It is a direct corollary of \cite[Problem 44A]{Willard1970topology} and
\cite[Theorem 1.1, Ch. 3]{kuipersNiederreiter1974}. For the reader's convenience, we provide a self-consistent proof in what follows.

\begin{lemma}\label{normal_stone_weierstrass}
Let $ \mathcal{V} $ be a class of functions from $\mathcal{C}_b(X)$. If $\mbox{sp}(\mathcal{V})$ contains the constant functions and separates closed sets. Then $ \mathcal{V}$  is a convergence-determining class with respect to $\mu$ on $X$.
\end{lemma}
\begin{proof}
By
\cite[Theorem 1.1, Ch.3]{kuipersNiederreiter1974}, we only need to prove that $\mbox{sp}(\mathcal{V})$ is dense in $\mathcal{C}_b(X)$.

For any $f \in \mathcal{C}_b(X)$, since $ f $ is bounded, $ c_0 = \|f\|_\infty = \sup_{x \in X} |f(x)| < \infty $. Our goal is to construct a sequence $ \{g_n\} \subseteq  \mbox{sp}(\mathcal{V})$ such that the partial sums $ G_n = \sum_{k=0}^n g_k $ converge uniformly to $ f $. Clearly, if $c_0=0$, then $f$ is a constant function and the conclusion holds trivially since the constant function $ 0 \in \mbox{sp}(\mathcal{V})$. In what follows we assume $c_0>0$.

Set
$$E_0 = \left\{ x \in X \,\Big|\, f(x) \geq \frac{c_0}{3} \right\}, \quad F_0 = \left\{ x \in X\,\Big|\,f(x) \leq - \frac{c_0}{3} \right\}.$$
Clearly, $ E_0 $ and $ F_0 $ are disjoint.
Since $\mbox{sp}(\mathcal{V})$ separates closed sets, there exists a function $ h_0 \in \mbox{sp}(\mathcal{V})$ such that $ h_0|_{E_0} = 1 $, $ h_0|_{F_0} = 0 $, and $ 0 \leq h_0 \leq 1 $. Define
$$g_0 = \frac{2 c_0}{3} h_0 - \frac{c_0}{3}.$$
Since $\mbox{sp}(\mathcal{V})$ contains the constant functions and is closed under scalar multiplication and addition, $ g_0 \in \mbox{sp}(\mathcal{V})$.
By $ 0 \leq h_0(x) \leq 1 $, we have $ -\dfrac{c_0}{3} \leq g_0(x) \leq \dfrac{c_0}{3} $  for any $x\in X$ and,
$$g_0(x) = \frac{2 c_0}{3} \cdot 1 - \frac{c_0}{3} = \frac{c_0}{3}, \ \forall\ x\in E_0 \text{ and }  g_0(x) = \frac{2 c_0}{3} \cdot 0 - \frac{c_0}{3} = -\frac{c_0}{3},\ \forall\ x\in F_0.$$

Next, let us consider the function $ f - g_0 $. If $ x \in E_0 $, $g_0(x)=\dfrac{c_0}{3}\le f(x) $ and $ f(x) \leq c_0 $. Then we have
\begin{equation}\label{lemma3.1 eq1}
0 \leq f(x) - g_0(x) \leq c_0 - \frac{c_0}{3} = \frac{2 c_0}{3}, \ \forall \ x\in E_0.
\end{equation}
Similarly,
\begin{equation}\label{lemma3.1 eq2}
-\frac{2 c_0}{3} \leq f(x) - g_0(x) \leq 0,\ \forall \ x\in  F_0.
\end{equation}
If $ x \notin E_0 \cup F_0 $, then $ -\dfrac{c_0}{3}< f(x) < \dfrac{c_0}{3} $ and $ -\dfrac{c_0}{3} \leq g_0(x) \leq \dfrac{c_0}{3} $. Therefore,
\begin{equation}\label{lemma3.1 eq3}
|f(x) - g_0(x)| \leq |f(x)| + |g_0(x)| < \frac{c_0}{3} + \frac{c_0}{3} = \frac{2 c_0}{3},\ \forall\ x \notin E_0 \cup F_0.
\end{equation}
		
By \eqref{lemma3.1 eq1}-\eqref{lemma3.1 eq3}, we have
$$\|f - g_0\|_\infty \leq \frac{2 c_0}{3}.$$
Define $ \varphi_0 = f - g_0 $ with $ \|\varphi_0\|_\infty \leq \dfrac{2 c_0}{3} $. For $ n \geq 1 $, given $ \varphi_{n-1} $ with $ \|\varphi_{n-1}\|_\infty \leq \dfrac{2^n c_0}{3^n} $, set $ c_n = \|\varphi_{n-1}\|_\infty $ and define
$$E_n = \left\{ x\in X \,\Big|\, \varphi_{n-1}(x) \geq \dfrac{c_n}{3}  \right\}, \quad F_n = \left\{ x\in X \,\Big|\, \varphi_{n-1}(x) \leq -\dfrac{c_n}{3}  \right\}.$$
Since $ \varphi_{n-1} $ is continuous, $ E_n $ and $ F_n $ are closed and disjoint. There exists $ h_n \in \mbox{sp}(\mathcal{V})$ such that $ h_n|_{E_n} = 1 $, $ h_n|_{F_n} = 0 $, and $ 0 \leq h_n \leq 1 $. Set
$$g_n = \frac{2 c_n}{3} h_n - \frac{c_n}{3}.$$
Clearly, $g_n \in \mbox{sp}(\mathcal{V})$, $ g_n|_{E_n} = \dfrac{c_n}{3} $, $ g_n|_{F_n} = -\dfrac{c_n}{3} $, and $ -\dfrac{c_n}{3} \leq g_n \leq \dfrac{c_n}{3} $. Define $ \varphi_n = \varphi_{n-1} - g_n $. By the same arguments in the above, we obtain that
$$\|\varphi_n\|_\infty \leq \frac{2 c_n}{3} \leq  \frac{2^{n+1} c_0}{3^{n+1}}.$$
		
Define the partial sums $G_n = \sum_{k=0}^n g_k $. Since $ \mbox{sp}(\mathcal{V})$ is a subalgebra, $G_n \in \mbox{sp}(\mathcal{V})$. We have
$$
f - G_n  =  f - \sum_{k=0}^n g_k \notag  =  \varphi_n,
$$
which implies that
$$\|f - G_n\|_\infty = \|\varphi_n\|_\infty \leq \frac{2^{n+1} c_0}{3^{n+1}}\to 0\ \text{as } n\to \infty.
$$
Therefore, $\mbox{sp}(\mathcal{V})$ is dense in  $\mathcal{C}_b(X)$ with respect to the supremum norm. This completes the proof.
\end{proof}

The following characterizations of u.d.\ sequences on a normal space $(X,\tau)$ (which need not be compact) extend the corresponding results for compact Hausdorff spaces (see, Theorem 1.2 and Exercise 1.6 in \cite{kuipersNiederreiter1974}).

\begin{theorem}\label{thm:uniform_distribution_equivalence}
The following conditions are equivalent for a sequence $\{x_n\}_{n=1}^\infty$ in $ X $.
\begin{enumerate}[(i)]
\item the sequence $\{x_n\}_{n=1}^\infty$ is $ \mu $-u.d.\;
\item for all open sets $ D \subseteq X $,
$$
\varliminf_{N \to \infty} \frac{A(D; N)}{N} \geq \mu(D);
$$
\item for all closed sets $ C \subseteq X $,
$$
\varlimsup_{N \to \infty} \frac{A(C; N)}{N} \leq \mu(C);
$$
\item for all $ \mu $-continuity sets $ M$ of $X $,
$$\lim_{N \to \infty} \frac{A(M; N)}{N} = \mu(M).$$
\end{enumerate}
\end{theorem}
	
\begin{proof}		
(i)$\Rightarrow$ (ii): 	Assume, to the contrary, that there exists an open set $D$ such that
$$\varliminf_{N \to \infty} \frac{A(D; N)}{N} < \mu(D).$$
Then, there exists $\delta > 0$ and a subsequence $\{N_k\}_{k=1}^{\infty}$ with $N_k \to \infty$ such that
$$\frac{A(D; N_k)}{N_k} < \mu(D) - \delta$$
for all $k$. Since $\mu$ is regular, there exists a closed set $F \subseteq D$ with $\mu(D \setminus F) < \dfrac{\delta}{2}$. By the regularity of $T_4$ space, there exists a continuous function $f: X \to [0, 1]$ such that $f(x) = 1$ for $x \in F$ and $f(x) = 0$ for $x \notin D$. Then,
$$\int_X f \, d\mu \geq \mu(F) > \mu(D) - \frac{\delta}{2}.$$
Since $\{x_n\}_{n=1}^\infty$ is $\mu$-u.d.\, we have
\begin{equation}\label{eq 3.1}
\lim_{N \to \infty} \frac{1}{N} \sum_{n=1}^N f(x_n) = \int_X f \, d\mu > \mu(D) - \frac{\delta}{2}.
\end{equation}
By the definition of $ f $, $f(x_n) \leq \chi_D(x_n)$ and
$$\frac{1}{N_k} \sum_{n=1}^{N_k} f(x_n) \leq \frac{A(D; N_k)}{N_k} < \mu(D) - \delta,$$
which contradicts to \eqref{eq 3.1} when $k$ is large enough. This proves (ii).
		
(ii)$\Rightarrow$(iii):	For a closed set $C \subseteq X$, since $X \setminus C$ is open, we have by (ii),
$$
\varliminf_{N \to \infty} \frac{A(X \setminus C; N)}{N} \geq \mu(X \setminus C) = 1 - \mu(C).
$$
Noting that $A(X \setminus C; N) = N - A(C; N)$, we obtain
$$
\varliminf_{N \to \infty} \left(1 - \frac{A(C; N)}{N}\right) \geq 1 - \mu(C),
$$
which implies
$$
\varlimsup_{N \to \infty} \frac{A(C; N)}{N} \leq \mu(C).
$$
Thus, condition (iii) holds.
		
(iii)$\Rightarrow$ (iv): Let $M \subseteq X$ be a $\mu$-continuity set, then $\mu(\overline{M}) = \mu(M) = \mu(\mbox{int}M)$. By $\mbox{int}M \subseteq M \subseteq \overline{M}$, we have $A(\mbox{int}M; N) \leq A(M; N) \leq A(\overline{M}; N)$. Applying the assumption to the closed set $\overline{M}$, we have
\begin{equation}\label{th3.1 eq1}
\varlimsup_{N \to \infty} \frac{A(\overline{M}; N)}{N} \leq \mu(\overline{M}),
\end{equation}
and since $X \setminus \mbox{int}M$ is closed, we have
$$
\varlimsup_{N \to \infty} \frac{A(X \setminus \mbox{int}M; N)}{N} \leq \mu(X \setminus \mbox{int}M).
$$
It follows that,
\begin{equation}\label{th3.1 eq2}
\varliminf_{N \to \infty} \frac{A(\mbox{int}M; N)}{N} = 1 - \varlimsup_{N \to \infty} \frac{A(X \setminus \mbox{int}M; N)}{N} \geq 1 - \mu(X \setminus \mbox{int}M) = \mu(\mbox{int}M).
\end{equation}
Combining \eqref{th3.1 eq1} with \eqref{th3.1 eq2}  we obtain that,
\begin{eqnarray*}
\mu(\overline{M}) &\geq& \varlimsup_{N \to \infty} \frac{A(\overline{M}; N)}{N} \geq \varlimsup_{N \to \infty} \frac{A(M; N)}{N} \notag \nonumber\\
			&\geq& \varliminf_{N \to \infty} \frac{A(M; N)}{N} \geq \varliminf_{N \to \infty} \frac{A(\mbox{int}M; N)}{N} \geq \mu(\mbox{int}M).
\end{eqnarray*}
Consequently,
$$
\lim_{N \to \infty} \frac{A(M; N)}{N} = \mu(M).
$$

(iv)$\Rightarrow$(i): It is equivalent to prove that
\begin{equation}\label{eq:weak_conv_cond}
\lim_{N \to \infty}\int_{X}fd\nu_N=\int_{X}fd\mu, \quad \forall \ f\in \mathcal{C}_{b}(X),
\end{equation}
where $\nu_N=\sum_{k=1}^N \delta_{x_k}$ and $\delta_{x_k}$ is the Dirac measure at $x_k$ ($k\in \mathbb{N}$).
Specifically, for any Borel set $M$,
$$\nu_N(M) = \frac{1}{N} \sum_{n=1}^N \delta_{x_n}(M) =\frac{1}{N} \sum_{n=1}^N \chi_M(x_n) = \frac{A(M;N)}{N}.$$
Then, by condition (iv), for any $\mu$-continuity set $M$ we have
$$\nu_N(M) \to \mu(M)\ \text{as}\ N\to \infty.$$
		
Let $ f: X \to \mathbb{R} $ be a bounded continuous function with $ |f(x)| \leq K $ for some $ K > 0 $ and all $ x \in X $. Clearly,
$$\int_X f \, d\nu_N = \frac{1}{N} \sum_{n=1}^N f(x_n).$$

By Lemma \ref{lem:layer_cake},
\begin{equation}\label{th3.1 eq3}
\int_X f \, d\nu_N  =  \int_0^K \nu_N(\{ x \,|\, f(x) > t \}) \, dt - \int_0^K \nu_N(\{ x \,|\, f(x) < -t \}) \, dt,
\end{equation}
and
\begin{equation}\label{th3.1 eq4}
\int_X f \, d\mu  =   \int_0^K \mu(\{ x\,|\, f(x) > t \}) \, dt  - \int_0^K \mu(\{ x\,|\, f(x) < -t \}) \, dt.
\end{equation}
		
Define
$$
M_t^+ = \{ x \,|\, f(x) > t \}, \quad M_t^- = \{ x\,|\,f(x) < -t \},\ t \in [0, K].
$$
Since $ f $ is continuous and $ X $ is a $T_4$ space, the sets $ M_t^+ $ and $ M_t^- $ are open and
$$\partial M_t^+ \subseteq \{ x \,|\, f(x) = t \}, \quad \partial M_t^- \subseteq \{ x  \,|\, f(x) = -t \}.
$$
Since $ \mu $ is a finite measure, the set of $ t $ for which $ \mu(\partial M_t^+) > 0 $ or $ \mu(\partial M_t^-) > 0 $ is at most countable (see Example \ref{ex:mu-continuous set}) and  $ M_t^+ $ and $ M_t^- $ are $ \mu $-continuity sets for almost all $ t\in [0,K]$.
Then, by (iv), for all such $ t $,
$$\nu_N(M_t^+) \to \mu(M_t^+), \quad \nu_N(M_t^-) \to \mu(M_t^-),\ \text{as } N \to \infty.
$$

For each $ t $, the functions $ t \mapsto \nu_N(M_t^+) $ and $ t \mapsto \nu_N(M_t^-) $ are bounded by 1. By Lebesgue's dominated convergence theorem,  %
\begin{equation}\label{th3.1 eq5}
\lim_{N \to \infty} \int_0^K \nu_N(M_t^+) \, dt = \int_0^K \lim_{N \to \infty} \nu_N(M_t^+) \, dt = \int_0^K \mu(M_t^+) \, dt,
\end{equation}
and
\begin{equation}\label{th3.1 eq6}
\lim_{N \to \infty} \int_0^K \nu_N(M_t^-) \, dt = \int_0^K \lim_{N \to \infty} \nu_N(M_t^-) \, dt = \int_0^K \mu(M_t^-) \, dt.
\end{equation}
		
Combining \eqref{th3.1 eq3}-\eqref{th3.1 eq4} with \eqref{th3.1 eq5}-\eqref{th3.1 eq6}, we obtain
\begin{eqnarray*}
\lim_{N \to \infty} \int_X f \, d\nu_N & =& \lim_{N \to \infty} \left( \int_0^K \nu_N(M_t^+) \, dt - \int_0^K \nu_N(M_t^-) \, dt \right)  \nonumber\\
&=& \int_0^K \mu(M_t^+) \, dt - \int_0^K \mu(M_t^-) \, dt   \nonumber\\
&=& \int_X f \, d\mu.
\end{eqnarray*}
This proves \eqref{eq:weak_conv_cond}.
\end{proof}

The following result is a direct corollary of Theorem \ref{thm:uniform_distribution_equivalence}.
	
\begin{corollary}\label{Weyl's Theorem for Normal Spaces}
The class
$$\left\{ \chi_M \,|\, M \text{ is a } \mu\text{-continuity set} \right\} $$ is a convergence-determining class with respect to $\mu$.
\end{corollary}

\section{Countable convergence-determining class and construction
of u.d.\ sequences on Polish spaces}
	
In this section, we assume that $(X,\tau)$ is a Polish space  and $\mu$ is a regular, non-atomic probability measure on $(X,\mathcal{B}(X))$.
Note that for a general non-compact topological space, $\mathcal{C}_{b}(X)$ need not be separable and the classical Stone-Weierstrass theorem  is inapplicable for constructing a countable convergence-determining class within $\mathcal{C}_{b}(X)$. In what follows, the tightness of the Borel probability measure on Polish space (see Lemma \ref{lem:tight-prob}) is employed to prove the existence of a countable convergence-determining class contained in $b(X)$.

\begin{theorem} \label{thm:countable_continuous_conv_det}
If $(X,\tau)$ is a Polish space and $ \mu $ is a regular probability measure on $(X,\mathcal{B}(X)) $, then, there exists a countable convergence-determining class $ \mathcal{V} \subset b(X) $ with respect to $ \mu $.
\end{theorem}
	
\begin{proof}
By Lemma \ref{lem:tight-prob}, the probability measure $ \mu $ on $ X $ is tight. For each $i\in \mathbb{N}$, let $ K_i \subset X $ be a compact subset of $X$ such that $\mu(K_i) > 1-\frac{1}{i}$.

For each fixed $ i\in \mathbb{N} $, since $ K_i $ is compact, Lemma \ref{lem:separable_Cb} ensures the existence of a countable set $ \left\{g_k^{i}\right\}_{k=1}^\infty \subset \mathcal{C}_{b}(K_i) $ that is dense in $\mathcal{C}_{b}(K_i)$ under the supremum norm. By Lemma \ref{lem:tietze}, each $ g_k^{i} $ can be extended to a function $ f_k^{i} \in \mathcal{C}_{b}(X)$ such that $ f_k^{i}|_{K_i} = g_k^{i} $ and $ \|f_k^{i}\|_\infty = \|g_k^{i}\|_\infty $. Define
$$ \mathcal{U} = \left\{ f_k^{i} \,\Big|\, i, k \in \mathbb{N} \right\} ,$$
 which is a countable set.  For each $i\in \mathbb{N}$, denoted by $K_i^c$  the complement of set $K_i$. Let $ \mathcal{W} = \left\{ \chi_{K_{i}^{c}} \,\Big|\, i  \in \mathbb{N} \right\}$ and define $ \mathcal{V} = \mathcal{U} \cup \mathcal{W} $, which is a countable subset of $b(X)$. We show that $ \mathcal{V} $ is a countable convergence-determining class.

Consider a sequence $ \{x_n\}_{n=1}^\infty \subset X $ such that
\begin{equation}\label{eq:assumption_conv_v}
\lim_{N \to \infty} \frac{1}{N} \sum_{n=1}^{N} f(x_n) = \int_X f \, d\mu,\  \forall f\ \in \mathcal{V}.
\end{equation}
Let $ f \in C_b(X) $ and fix $ \epsilon > 0 $. Set $ M = \|f\|_\infty $. Choose $i$ such that
\begin{equation}\label{th4.1 eq2}
\dfrac{1}{i} < \dfrac{\epsilon}{8M + \epsilon } .
\end{equation}
Since $ \{g_k^{i}\}_{k=1}^\infty $ is dense in $ \mathcal{C}_{b}(K_i) $, there exists $ k_0 $ such that $ \sup\limits_{x\in K_i} |f(x) - g_{k_0}^{i}(x)| < \dfrac{\epsilon} {4} $. Let  $f_{k_0}^{i}$ be the extension of $g_{k_0}^{i}$ on $X$, then,  $ f_{k_0}^{i}|_K = g_{k_0}^{i} $,
\begin{equation}\label{th4.1 eq3}
\sup\limits_{x\in K_i} |f(x) - f_{k_0}^{i}(x)| < \dfrac{\epsilon} {4}, \text{ and }  \|f_{k_0}^{i}\|_\infty < M + \dfrac{\epsilon} {4}.
\end{equation}
		
Clearly,
\begin{equation}\label{eq:triangle_inequality}
\left| \frac{1}{N} \sum_{n=1}^{N} f(x_n) - \int_X f \, d\mu \right| \leq \left| \frac{1}{N} \sum_{n=1}^{N} (f - f_{k_0}^{i})(x_n) \right|
+ \left| \frac{1}{N} \sum_{n=1}^{N} f_{k_0}^{i}(x_n) - \int_X f_{k_0}^{i} \, d\mu \right|
 + \left| \int_X (f_{k_0}^{i} - f) \, d\mu \right|.			
\end{equation}

For the first term of the right-hand side of equation \eqref{eq:triangle_inequality},
\begin{eqnarray}\label{eq:first_term_split}
\left| \frac{1}{N} \sum_{n=1}^{N} (f - f_{k_0}^{i})(x_n) \right|
&\le& \frac{1}{N} \sum_{n=1}^{N} |f(x_n) - f_{k_0}^{i}(x_n)| \chi_{K_i}(x_n) + \frac{1}{N} \sum_{n=1}^{N} |f (x_n)- f_{k_0}^{i}(x_n)| \chi_{K_i^c}(x_n) \notag \\
&\leq& \sup_{x\in K_i} |f(x) - f_{k_0}^{i}(x)| + ( \|f\|_\infty + \|f_{k_0}^{i}\|_\infty ) \frac{1}{N} \sum_{n=1}^{N} \chi_{K_i^c}(x_n).
\end{eqnarray}
Since $ \chi_{K_i^c} \in  \mathcal{V} $,
\begin{equation}\label{eq:indicator_limit_conv}
\lim_{N \to \infty} \frac{1}{N} \sum_{n=1}^{N}\chi_{K_i^c}(x_n) = \int_X \chi_{K_i^c} \, d\mu = \mu(K_i^c) < \frac{1}{i}.
\end{equation}
		
Combining \eqref{eq:first_term_split}-\eqref{eq:indicator_limit_conv} with \eqref{th4.1 eq2}-\eqref{th4.1 eq3}, we have
\begin{eqnarray}\label{eq:first_term_limsup}
\lim_{N \to \infty} \left| \frac{1}{N} \sum_{n=1}^{N} (f - f_{k_0}^{i})(x_n) \right| &\leq& \sup_{x\in K_i} |f(x) - f_{k_0}^{i}(x)| + ( \|f\|_\infty + \|f_{k_0}^{i}\|_\infty ) \mu(K_i^c) \notag \\
&<& \frac{\epsilon}{4} +  (2M  + \dfrac{\epsilon}{4})\dfrac{1}{i} \notag \\
&<& \frac{\epsilon}{2}.	
\end{eqnarray}

For the second term of the right-hand side of equation \eqref{eq:triangle_inequality}, since $ f_{k_0}^{i}\in \mathcal{V} $,
\begin{equation}\label{eq:second_term_limit}
\lim_{N \to \infty} \left| \frac{1}{N} \sum_{n=1}^{N} f_{k_0}^{i}(x_n) - \int_X f_{k_0}^{i} \, d\mu \right| = 0.			
\end{equation}
		
For the third term of the right-hand side of equation \eqref{eq:triangle_inequality},
\begin{eqnarray}
\left| \int_X (f_{k_0}^{i} - f) \, d\mu \right| &\le& \left| \int_{K_i} (f_{k_0}^{i} - f) \, d\mu\right| + \left|\int_{K_{i}^c} (f_{k_0}^{i} - f) \, d\mu \right| \notag \\
&\leq& \sup_{x\in K_i} |f_{k_0}^{i}(x)- f(x)| + ( \|f_{k_0}^{i}\|_\infty + \|f\|_\infty ) \mu(K_{i}^c) \notag \\
&<&  \frac{\epsilon}{4} +  (2M  + \dfrac{\epsilon}{4})\dfrac{1}{i} \notag \\
&<&  \frac{\epsilon}{2}.
\label{eq:third_term_bound}
\end{eqnarray}

Combining \eqref{eq:triangle_inequality}, \eqref{eq:first_term_limsup}, \eqref{eq:second_term_limit}, and \eqref{eq:third_term_bound}, we have
$$\lim_{N \to \infty} \left| \frac{1}{N} \sum_{n=1}^{N} f(x_n) - \int_X f \, d\mu \right| \leq \epsilon.			
$$
By the arbitrariness of $\epsilon$, we have
$$\lim_{N \to \infty} \left| \frac{1}{N} \sum_{n=1}^{N} f(x_n) - \int_X f \, d\mu \right|=0.			
$$
Finally, by the arbitrariness of $f\in \mathcal{C}_{b}(x)$, $\{x_n\}_{n=1}^{\infty}$ is an u.d.\ sequence. This proves that $\mathcal{V}$ is  a convergence-determining class.
\end{proof}

Notice that,  the number-theoretic approaches cannot be used to construct  (low-discrepancy) u.d.\ sequences in general non-compact topological space $X$. In what follows, we demonstrate that,  the Strong LLN can be employed to construct u.d.\ sequences on Polish space by virtue of the existence of a countable convergence-determining class.

\begin{corollary}\label{slln-uniform}
Let $(\Omega, \mathcal{F}, P)$ be a probability space, and $\xi_n: \Omega \to X$, $n=1,2,\ldots$,  be a sequence of i.i.d.\ random variables such that the induced probability distribution of each $\xi_n$ on $X$ coincides with $\mu$. Then there exists a set $ V' \subset \Omega $ with $ P(V') = 1 $ such that for each $ \omega \in V' $, the sample path $ \{\xi_n(\omega)\}_{n=1}^\infty $ is $\mu$-u.d.\ on $ X $.
\end{corollary}

\begin{proof}
Let $\mathcal{V} \subseteq b(X)$ be the countable convergence-determining class defined in Theorem \ref{thm:countable_continuous_conv_det}.
By the Strong LLN, for each $f \in \mathcal{V}$, the set $V_f \subset \Omega$ defined by
$$
V_f = \left\{ \omega \in \Omega \ \middle|\ \lim_{N \to \infty} \frac{1}{N} \sum_{n=1}^{N} f(\xi_n(\omega))=\mathbb{E}f(\xi_1)= \int_X f \, d\mu \right\}
$$
satisfies $P(V_f) = 1$.
Define
$$
V' = \bigcap_{f \in \mathcal{V}} V_f.\label{eq:define_V_prime_general}
$$
Since $\mathcal{V}$ is countable,  $P(V') = 1$. Then, for any $\omega \in V'$, the sequence $\{\xi_n(\omega)\}_{n=1}^\infty$ satisfies
$$\lim_{N \to \infty} \frac{1}{N} \sum_{n=1}^{N} f(\xi_n(\omega)) = \int_X f \, d\mu, \quad \forall\ f \in \mathcal{V}.$$
By the definition of  convergence-determining class, the sequence $\{\xi_n(\omega)\}_{n=1}^\infty$ is $\mu$-u.d.\ on $X$.
\end{proof}

\begin{remark}
The existence of the probability space $(\Omega, \mathcal{F}, P)$ and the sequence of independent random variables $\xi_1, \xi_2, \dots$ satisfying the hypotheses of Corollary \ref{slln-uniform} can be established via Kolmogorov  Extension Theorem (see \cite[Theorem 2.1.21]{Durrett2019Prob}).

Indeed, let $\Omega = X^{\infty}$, the space of infinite sequences $x = (x_1, x_2, \dots)$ where each $x_n \in X$, and let $\mathcal{F}$ be the Borel $\sigma$-algebra on $X^{\infty}$ endowed with the product topology induced by the topology on $X$. Define the probability measure $P$  on $(\Omega, \mathcal{F})$ to be infinite product measure $\mu^{\infty}$, i.e., for any finite-dimensional cylinder set $A_1 \times \cdots \times A_n \times X \times X \times \cdots \in \mathcal{F}$ with $A_i \in \mathcal{B}(X)$,
$$
P(A_1 \times \cdots \times A_n \times X \times X \times \cdots) = \mu(A_1) \cdot \mu(A_2) \cdots \mu(A_n).
$$
By Kolmogorov's Extension Theorem, the probability space $(\Omega,\mathcal{F}, P)$ is well-defined.
For each $n\in \mathbb{N}$, define the random variable  $\xi_n: \Omega \to X$ by
$$\xi_n(\omega) = x_n,\quad \forall\ \omega = (x_1, x_2, \dots) \in \Omega.
$$
It is straightforward to verify that  $(\Omega, \mathcal{F}, P)$ and  $\{\xi_n\}_{n=1}^{\infty}$ satisfy all the conditions stated   in Corollary \ref{slln-uniform}. For more details, we refer the readers to 
\cite[Section 2.1.4]{Durrett2019Prob}.
\end{remark}

Corollary \ref{slln-uniform} provides an approach to generating u.d.\ sequences on Polish space $(X, \mathcal{B}(X), \mu)$ by strong LLN. In what follows we demonstrate that, for any ergodic transformation $T$, the  orbits of $T$ are $u$-u.d.\ on a full-measure set. First, we establish the existence of an ergodic transformation on a Polish space equipped with a non-atomic probability measure in the lemma below.
	
\begin{lemma}\label{thm:ergodic-existence}
There exists an ergodic transformation $T$ on $(X, \mathcal{B}(X), \mu)$.
\end{lemma}
	
	% Providing a detailed proof with clear formatting
\begin{proof}
Since $X$ is a Polish space, by Lemma \ref{lem:atomless_isomorphism}, there exists a Borel isomorphism $f: X \to [0,1]$ such that the measure space $(X, \mathcal{B}(X), \mu)$ is isomorphic to $([0,1], \mathcal{B}([0,1]), \lambda|_{[0,1]})$ and, for all $A \in \mathcal{B}([0,1])$, $\mu(f^{-1}(A)) = \lambda|_{[0,1]}(A)$.

Let $S$ be an arbitrary ergodic transformation on $([0,1], \mathcal{B}([0,1]), \lambda|_{[0,1]})$. Define  $T: X \to X$ by
$$
T = f^{-1} \circ S \circ f,
$$
i.e., $T(x) = f^{-1}(S(f(x)))$ for any $x\in X$. We prove that $T$ is an ergodic transformation.
		
First, we prove that $T$ preserves $\mu$. For any Borel set $A \in \mathcal{B}(X)$, by the measure-preserving property of $f$ and $S$,
\begin{eqnarray*}
\mu(T^{-1}(A))&=& \mu\left( f^{-1} \left( S^{-1} \left( f(A) \right) \right) \right)\\
 &=& \lambda|_{[0,1]}\left( S^{-1} \left( f(A) \right) \right)\\
&=&\lambda|_{[0,1]}(f(A))\\
&=&\mu(A).
\end{eqnarray*}
This confirms that $T$ is a measure-preserving transformation.

Next, we verify the ergodicity of $T$. Suppose that there exists a Borel set $A \in \mathcal{B}(X)$ such that $T^{-1}(A) = A$. We show that $\mu(A) = 0$ or $\mu(A) = 1$.
		
Define $B = f(A) \subset [0,1]$, then $B \in \mathcal{B}([0,1])$ and
$$S^{-1}(B)  = S^{-1}(f(A))= f\left( T^{-1}(A) \right).$$
Since $T^{-1}(A) = A$, we have
$$	f\left( T^{-1}(A) \right) = f(A) = B.$$
Consequently, $S^{-1}(B) = B.$	Since $S$ is ergodic with respect to $\lambda|_{[0,1]}$, 		
$$\lambda|_{[0,1]}(B) = 0 \quad \text{or} \quad \lambda|_{[0,1]}(B) = 1.$$
By the measure-preserving property of $f$, $\mu(A) = \lambda|_{[0,1]}(f(A)) = \lambda|_{[0,1]}(B)$, which is equal to $0$ or $1$.
This proves that $T$ is an ergodic transformation and the proof is complete.
\end{proof}

The following result, which is a direct corollary of Birkhoff's ergodic theorem and Theorem \ref{thm:countable_continuous_conv_det}, indicates that the orbits of an ergodic transformation are $\mu$-u.d.\ on a full-measure set.

\begin{corollary} \label{uniformly_distributed}
Let $T$ be an ergodic transformation on $(X, \mathcal{B}(X), \mu)$. Then there exists a set $ W' \subset X $ with $\mu(W') = 1$ such that for each $ x \in W' $, the orbit $\{T^n(x)\}_{n=0}^\infty$ of the ergodic transformation is $\mu$-u.d.\ on $ X $.
\end{corollary}
	
\begin{proof}
By Theorem \ref{thm:countable_continuous_conv_det}, there exists a countable convergence-determining class $\mathcal{V} \subseteq b(X)$ with respect to $\mu$.  For each $ f \in \mathcal{V}$, define
$$
W_f = \left\{ x \in X \ \middle|\ \lim_{N \to \infty} \frac{1}{N} \sum_{n=0}^{N-1} f(T^n(x)) = \int_X f \, d\mu \right\}.
$$
By Lemma \ref{lem:birkhoff}, $\mu(W_f) = 1$.  Let
$$
W' = \bigcap_{f \in \mathcal{V}} W_f.
$$
Since $\mathcal{V}$ is countable, we have $\mu(W') = 1$. For any $x \in W'$ and all $f \in \mathcal{V}$, $x \in W_f$  and hence
$$\lim_{N \to \infty} \frac{1}{N} \sum_{n=0}^{N-1} f(T^n(x)) = \int_X f \, d\mu, \quad \forall\ f \in \mathcal{V}.$$
It follows from the definition of  convergence-determining class that  $\{T^n(x)\}_{n=0}^\infty$ is $\mu$-u.d.\ on $X$.
\end{proof}

\section{Concluding remarks}

In this paper, we establish a theoretical framework for u.d.\ sequences on non-compact topological spaces, with a particular focus on $\mu$-u.d.\ sequences on $T_4$ and Polish spaces. Specifically, we prove the existence of a countable convergence-determining class of $\mu$-u.d.\ sequences on Polish spaces, and indicate that, under the existence of convergence-determining class, random sampling and ergodic transformations can be employed to construct $\mu$-u.d.\ sequences on non-compact topological spaces. These results not only advance the pure theory of uniform distribution but also provide a new approach for the development of numerical integration methods in infinite-dimensional settings.

In contrast to the well-developed theory of uniform distribution on compact spaces, particularly uniform distribution mod $1$ in finite-dimensional settings, the study of uniform distribution on general non-compact spaces remains in its infancy.
Based on the results established in this paper, the following topics merit further investigation.

\begin{itemize}
  \item[(1)] The construction of u.d.\ sequences in more concrete topological linear spaces, such as separable Hilbert spaces and the space of continuous functions defined on the interval $[0,1]$, both of which are Polish spaces;

  \item[(2)] The characterization of discrepancy for u.d.\ sequences in infinite-dimensional spaces, as well as the sharp bound of such discrepancies;

  \item[(3)] The perturbation stability of u.d.\ sequences in infinite-dimensional spaces and its applications to error analysis in numerical integration.
\end{itemize}

%{\bf Acknowledgements.}

%\appendix

%\section{Appendix}

%\section*{Acknowledgement}

\end{document}